\documentclass{amsart}[12pt]

 \usepackage{graphicx,amssymb,amstext,amsmath}
 \usepackage{enumerate}
 \usepackage{amsthm}
\usepackage{xcolor}
\usepackage[colorlinks=true,
linkcolor=blue,
urlcolor=blue,
citecolor=blue]{hyperref}
\usepackage[margin = 4cm]{geometry}
\usepackage{listings}\usepackage{float}

 \newtheorem{theorem}{Theorem}[section]
 \newtheorem{thm}[theorem]{Theorem}
 \newtheorem{obs}[theorem]{Observation}
 
 \newtheorem{lem}[theorem]{Lemma}
 \newtheorem{cor}[theorem]{Corollary}
 \newtheorem{prop}[theorem]{Proposition}
 \newtheorem{quest}[theorem]{Question}
 \newtheorem{conj}[theorem]{Conjecture}

\DeclareMathOperator\one{\bf{1}}
\DeclareMathOperator\fld{\mathbb{F}}
\DeclareMathOperator\flde{\mathbb{E}}

\DeclareMathOperator\Ind{ind}
\DeclareMathOperator\Res{res}

\DeclareMathOperator\der{der}
\DeclareMathOperator\id{id}

\DeclareMathOperator\fix{fix}
\DeclareMathOperator\sym{Sym}

\DeclareMathOperator\AGL{AGL}
\DeclareMathOperator\GL{GL}
\DeclareMathOperator\SL{SL}
\DeclareMathOperator\PGL{PGL}
\DeclareMathOperator\PSL{PSL}
\DeclareMathOperator\stab{\agl{q}_B}
\DeclareMathOperator{\ag}{AG}

\newcommand\agl[1]{\AGL(2, #1)}
\newcommand\gl[1]{\GL(2, #1)}
\newcommand\slg[1]{\SL(2, #1)}
\newcommand\pgl[1]{\PGL(2, #1)}
\newcommand\psl[1]{\PSL(2, #1)}

\begin{document}
\title[EKR results for some linear groups]{Erd\H{o}s-Ko-Rado results for the general linear group, the special linear group and the affine general linear group}

\author[K.~Meagher]{Karen Meagher${^*}$ }
\address{Department of Mathematics and Statistics, University of Regina, Regina, Saskatchewan
S4S 0A2, Canada}
\email{meagherk@uregina.ca}
\thanks{${^*}$Research supported in part by an NSERC Discovery Research Grant,
    Application No.: RGPIN-2018-03952.}

\author[A.~Sarobidy Razafimahatratra]{A. Sarobidy Razafimahatratra}
\address{Department of Mathematics and Statistics, University of Regina, Regina, Saskatchewan
S4S 0A2, Canada}
\email{sarobidy@phystech.edu}

\begin{abstract}
In this paper, we show that both the general linear group $\gl{q}$ and the special linear group $\slg{q}$ have both the EKR property and the EKR-module property. This is done using an algebraic method; a weighted adjacency matrix for the derangement graph for the group is found and Hoffman's ratio bound is applied to this matrix. We also consider the group $\agl{q}$ and the 2-intersecting sets in $\PGL(2,q)$.
\end{abstract}
    


\maketitle

\section{Introduction}\label{introduction}

In this paper we will examine some Erd\H{o}s-Ko-Rado properties for the general linear
group $\gl{q}$, the special linear group $\slg{q}$ and the affine general linear group $\agl{q}$. 
The Erd\H{o}s-Ko-Rado (EKR) theorem for permutation groups has to do with finding maximum sets of permutations within a group so that any two permutations in the set \textsl{intersect}. Let $G\leq \sym(n)$ be a transitive group of degree $n$ and acting on the set $[n]:= \left\{ 1,2, \ldots, n \right\}$. Permutations $g$ and $h$ of $G$ \textsl{intersect} if there exists some $i \in  \{1,\dots, n\}$ such that $g(i) = h(i)$, or equivalently, 
$h^{-1} g (i) = i$. If $g$ and $h$ do not intersect, then $h^{-1} g$ is a \textsl{derangement}.

For any positive integers $n$ and $k$ such that $2k\leq n$, the original EKR theorem~\cite{erdos1961intersection} gives the size of the largest collection of $k$-uniform intersecting sets of $[n]$. The maximum size is achieved by taking all subsets that contain a fixed point. These collections are called the \textsl{canonical} intersecting set systems. Analogously, the \textsl{canonical intersecting sets of permutations} from a transitive group $G\leq \sym(n)$ are the sets of all permutations of $G$ that map some $i$ to some $j$, where $i,j\in [n]$. The stabilizer in a group $G$ of a point $i$ will be denoted by $G_i$. The canonical intersecting sets in $G$ are the sets
\[
S_{i,j}:=\{g\in G\; : \; g(i)=j\};
\]
these are exactly the cosets of the point-stabilizers $G_i$. These are clearly
intersecting sets of permutations and it is easy to see that if $G$ is
a transitive group with degree $n$, then $|S_{i,j}|=\frac{|G|}{n}$. 

If $S$ is an intersecting set in $G$, then for any element $g \in G$ the set $gS$ is also intersecting set; in particular, if $S$ is a canonical intersecting set, then $gS$ is also a canonical intersecting set. 
The \textsl{characteristic vector} of a set $S$ of permutations in a group $G$ is a length-$|G|$ vector 
with entries indexed by the elements of the group and the $g$-entry is 1 if $g \in S$ and 0 otherwise. 
We will denote the characteristic vector of $S_{i,j}$ by $v_{i,j}$ (assuming that the
group $G$ is clear from context). The group $G$ acts on these vectors by its action on the indices. Under this action, the vector space spanned by the characteristic vectors of the canonical intersecting sets in a group $G$ is a $\mathbb{C}[G]$-module. We call this module the \textsl{EKR-module}.

In this paper we consider the following three properties of a finite transitive group.

\begin{enumerate}
\item {\bf EKR-property}: A group has this property if the size of a
  maximum intersecting set of permutations is the size of a canonical
  intersecting set. 
\item {\bf EKR-module property}: A group has this property if the characteristic vector of any maximum intersecting set is in the EKR-module.
\item {\bf Strict-EKR property}: A group has the strict EKR property if the canonical intersecting sets are the only maximum intersecting sets.
\end{enumerate}

The first result on the EKR-property for permutations dates back to the 1977 paper of Deza and Frankl~\cite{Frankl1977maximum}. The EKR-property and the strict-EKR property have been an active research area over the past decade~\cite{ahmadi2014new, ellis2012setwise, ellis2011intersecting, godsil2009new, meagher2011erdHos, spiga2019erdHos}.  The EKR-module property is a fairly new concept introduced by Ahmadi and Meagher in~\cite{meagher2015erdos}. We state an observation on the EKR-module property.

\begin{obs}
Let $G$ be a transitive group. If $G$ has the EKR-module property, then the characteristic vector of any maximum coclique is a linear combination of the characteristic vectors of the canonical coclique.
\end{obs}

It was shown in~\cite{meagher2016erdHos} that any 2-transitive group has the EKR-property, and in~\cite{meagher2021all} it was further shown that any 2-transitive group has the EKR-module property. There are several papers showing different 2-transitive groups have the strict-EKR property \cite{godsil2009new,meagher2011erdHos,spiga2019erdHos}. In many of these papers, the result follows using an algebraic method that relies on the 2-transitivity of the group. In this paper we will show that this approach can be used for $\gl{q}$ and $\slg{q}$ (which are transitive but not $2$-transitive), where $q$ is a prime power.

\begin{thm}
	For any prime power $q$, the groups $\gl{q}$ and $\slg{q}$ with their natural action on $\mathbb{F}_q^2 \setminus \{0\}$ have the EKR-property and the EKR-module property. Moreover, $\gl{q}$ does not have the strict-EKR property.
\end{thm}

The EKR-property for $\gl{q}$ was first proved in~\cite{MR3188504}. However, this paper claimed that $\gl{q}$ also has strict-EKR property, unfortunately the proof of this contained an error, as there is a maximum intersecting set of $\gl{q}$ which is a subgroup and is not equal to a point-stabilizer (a corrected version should appear soon~\cite{ahanjidehCommunication}). 
The proof in~\cite{MR3188504} relies on the existence of Singer subgroups of $\gl{q}$, i.e., regular subgroups of $\gl{q}$, to prove the EKR-property. Our proof also uses the Singer subgroups, along with Hoffman's ratio bound on a weighted adjacency matrix. This approach puts into perspective the deep algebraic properties behind the EKR property. A similar method is applied to $\slg{q}$.

We also consider the affine group $\agl{q}$, where $q$ is a prime power.  The action of $\agl{q}$ on the points of the affine plane is $2$-transitive. Therefore, by~\cite{meagher2021all} and~\cite{meagher2016erdHos}, it has both the EKR-property and the EKR-module property. In this paper, we prove however that the action of $\agl{q}$ on the lines of the affine plane does not have the EKR property.

\begin{thm}
	For any prime power $q$, the group $\agl{q}$ acting on the lines of the affine plane does not have the EKR property. However, if $\mathcal{F} \subset \agl{q}$ is intersecting, then $|\mathcal{F}|\leq q^3(q-1)^2$.\label{thm:AGL}
\end{thm}

\section{Module Method}

Throughout this section, we let $G$ be a transitive permutation group of degree $n$. The \textsl{derangement graph} of $G$,
denoted by $\Gamma_G$, has the elements of $G$ as its vertices and two
vertices are adjacent if and only if they are not intersecting. The graph $\Gamma_G$ is the Cayley graph of $G$, with connection set equal to the set of all derangements of $G$; we denote this set by $\der(G)$. The derangement graph $\Gamma_G$ is defined so 
that a set $S$ of $G$ forms a coclique in $\Gamma_G$ if and only if $S$ is an intersecting set of $G$.

The \textsl{adjacency matrix} $A(X)$ of a graph $X$ is a matrix in which rows and columns are indexed by the vertices of $X$ and the $(i ,j)$-entry is 1 if $i\sim j$, and 0 otherwise. The \textsl{eigenvalues} of the graph are the eigenvalues of its adjacency matrix. 

For any group $G \leq \sym(n)$, the connection set of  the Cayley graph $\Gamma_G$ is the union of all the conjugacy classes of derangements of $G$, so $\Gamma_G$ is a \textsl{normal} Cayley graph. The eigenvalues of this
graph can be calculated using the irreducible representations of $G$. This follows from Babai's formula for the eigenvalues of normal Cayley graphs \cite{babai1979spectra} (this is also in Diaconis and Shahshahani~\cite{diaconis1981generating}). 

\begin{thm}\label{Diaconis} 
Let $G$ be a group. The eigenvalues of the Cayley graph $\Gamma_G$ are given by
\[
\eta_{\chi}=\frac{1}{\chi(\id)}\sum_{x \in \der(G)}\chi(x),
\]
where $\chi$ ranges over all irreducible characters of $G$. Moreover,
the multiplicity of the eigenvalue $\lambda$ is
$\sum_{\chi}\chi(\id)^2$, where the sum is taken over all irreducible
representations $\chi$ with $\eta_\chi = \lambda$.
\end{thm}

The largest eigenvalue of $\Gamma_G$ is $|\der(G)|$. This is the valency of the derangement graph and it is afforded by the trivial character. Once we have the eigenvalues of a Cayley graph, we can apply the ratio bound to find an upper bound on the size of a maximum coclique, this bound is also known as Hoffman's bound and Delsarte's bound (a history of this bound can be found in the lovely paper by Haemers~\cite{haemers2021hoffman}). 
We state the version of this bound for a weighted graph, since this is the version that we will be using.

A \textsl{weighted adjacency matrix} $A_{W}(X)$ for a graph $X$ is a symmetric matrix, with zeros on the main diagonal, in which rows and columns are indexed by the vertices and the $(i, j)$-entry is $0$ if $i\not\sim j$. A weighted adjacency matrix for a graph can be thought of as a weighting on the edges of the graph, we note that is, the weight of an edge can be equal to $0$. For a $d$-regular graph, or for any weighted adjacency matrix with constant row sum equal to $d$, the all-ones vector $\mathbf{1}$ is an eigenvector with eigenvalue $d$. 

\begin{theorem}[Delsarte-Hoffman bound or ratio bound \cite{godsil2016erdos,haemers2021hoffman}]\label{ratioBound}
Let $A$ be a weighted adjacency matrix for a graph $X$ on vertex set $V(X)$. If $A$ has constant row sum $d$ and least eigenvalue $\tau$, then 
\begin{equation*}\label{RatioBound}
\alpha(X)\leq \frac{|V(X)|}{1-\frac{d}{\tau}}.
\end{equation*}
If equality holds for some coclique $S$ with characteristic vector $\nu_{S}$, then 
\begin{equation*}
\nu_{S}-\frac{|S|}{|V(X)|}\mathbf{1}
\end{equation*}
is an eigenvector with eigenvalue $\tau$.
\end{theorem}

For any group $G$, the derangement graph $\Gamma_G$ is the union of graphs in the \emph{conjugacy class association scheme}; details can be found in~\cite{godsil2016algebraic}.
The vertex set for this association scheme is $G$, hence it has $|G|$ vertices, and has rank equal to the number of conjugacy classes of $G$. If $C_1,C_2,\ldots,C_k$ are the non-trivial conjugacy classes of $G$ (so $C_i \neq \{id\}$, for any $i\in \{1,2,\ldots,k\}$), then we let $A_i$ be the adjacency matrix of the Cayley digraph $\operatorname{Cay}(G,C_i)$, for any $i\in \{1,2,\ldots,k\}$. Let $C_{i_1},C_{i_2},\ldots,C_{i_\ell}$ be the conjugacy classes of derangements of $G$. We conclude that the adjacency matrix of $\Gamma_G$ is 
\begin{align*}
	A(\Gamma_{G}) = \sum_{j = 1}^\ell A_{i_j}. 
\end{align*}
Since the conjugacy class association scheme of $G$ is a commutative association scheme, the matrices $\{A_i \, | \, i =0,\dots ,d\}$ are simultaneously diagonalizable. Each of the common eigenspaces is a union of irreducible $\mathbb{C}[G]$-modules; so the eigenvalues can be found using the irreducible representations of $G$. If $\chi$ is an irreducible representation of $G$, then the eigenvalue of $A_i$ belonging to $\chi$ is $\frac{|C_i|}{\chi(\id)}\chi(x_i)$ where $x_i$ is any element in $C_i$.

If a weighted adjacency matrix of a Cayley graph on $G$ has the form $A = \sum_{i} a_i A_i$ (so the weights are constant on the conjugacy classes), 
then the eigenvalues of $A$ are 
\[
\frac{1}{\chi(\id)} \sum_{i} a_i |C_i| \chi(x_i).
\]
If the weights on all the conjugacy classes of derangements are 1 and 0 for all other conjugacy classes, then this equation is the formula in Theorem~\ref{Diaconis}.

The \textsl{permutation character} of $G$ is the character given by 
\[
\fix(g) = |  \{ i\in [n] : i^g = i \} |,
\]
for $g \in G$. This is equal to the representation induced by the trivial representation on the stabilizer of a point
\[
\Ind(1_{G_i})^{ G }(g)  = \fix(g).
\]

The representation corresponding to the permutation character is called the \emph{permutation representation}.
We call the $\mathbb{C}[G]$-module of the permutation representation the \textsl{permutation module}.
A group is $2$-transitive if and only if the permutation representation is the sum of two irreducible representations, namely the trivial representation and another irreducible representation.  Thus, for a 2-transitive group the character $\chi(g) = \fix(g)-1$ is irreducible. In this case the eigenvalue of the derangement graph afforded by $\chi$ is 
\[
\eta_{\chi}=\frac{-|\der(G) |}{n-1},
\]
(see~\cite{ahmadi2015erdHos} for details).
If $G$ is 2-transitive then it is a straightforward calculation to show that the dimension of the permutation module is $1+(n-1)^2$ and is spanned by $\{ v_{i,j} : i, j \in [n] \}$ (again, see~\cite{ahmadi2015erdHos} for details). So, in this case, the permutation module is isomorphic to the EKR-module. 
Further, every 2-transitive group has the EKR-module property~\cite{meagher2021all}. This means that for any 2-transitive group the characteristic vector of the largest intersecting set is a linear combination of the characteristic vectors of the canonical intersecting sets.

For any simply transitive group (i.e., transitive groups that are not $2$-transitive), the permutation module is still a $\mathbb{C}[G]$-module, but it is the sum of more than two irreducible $\mathbb{C}[G]$-modules. In many cases the eigenvalues for the non-trivial irreducible modules in the permutation module are not equal, and the ratio bound does not hold with equality.

The goal of this paper is to demonstrate how the algebraic approach can be applied to some simply transitive groups, namely general linear group and the special linear group.  Our plan is to weight the adjacency matrix of the graph $\Gamma_G$, for $G=\gl{q}$ and $\slg{q}$, so that the eigenvalues for all non-trivial irreducible $\mathbb{C}[G]$-modules in the permutation module are equal and the ratio bound holds with equality. Then, we can conclude the group has the EKR property, and we can further show that the characteristic vector for any maximum coclique is in the EKR-module. This approach has been used for other simply transitive groups, such as the transitive action of $\sym(n)$ on both ordered and unordered tuples in~\cite{ellis2012setwise, ellis2011intersecting}, for $n$ sufficiently large. This approach is also effectively used for the action of $\sym(n)$ on pairs~\cite{meagher20202} and 3-sets~\cite{behajaina20203} of $[n]$ for all $n\geq 5$. 

We will assign weights to the conjugacy classes of derangements in the group; the goal is to find a weighting to get the best bound from the ratio bound. We will form a linear program, in which the eigenvalues from all but the trivial representation are greater than or equal to -1 and then we maximize the eigenvalue corresponding to the trivial representation (this will be the largest eigenvalue). We consider the following setup.

\noindent{\bf Linear program.}
	Let $G\leq \sym(n)$ be a transitive permutation group with conjugacy classes of derangements $D_1,D_2,\ldots,D_k$ and trivial character $\chi_0$. Let $g_1,g_2,\ldots,g_k$ be representatives of the conjugacy classes $D_1,D_2,\ldots,D_k$, respectively. We consider the following optimization problem.
	\begin{align}
	\begin{split}
	\mathsf{Maximize } &\ \ \ \lambda_{\chi_0} = \displaystyle\sum_{i=1}^k \omega_i |D_i|,\\
	\mathsf{Subject \, to } & 	\\
	& 
	\begin{aligned}
	&\lambda_{\chi} = \displaystyle\sum_{i=1}^k \omega_i |D_i| \chi(g_i) \geq -1, \quad & & \forall \chi \in \operatorname{Irr}(G) \setminus \{\chi_0\}, \\  
	&\omega_i \in \mathbb{R}, \quad & & \forall i\in \{1,2,\ldots,k\}.
	\end{aligned}
	\end{split}\label{general-linear-program}
	\end{align}
If the solution of the linear programming (LP) problem \eqref{general-linear-program} is equal to $n-1$, then applying the ratio bound, we have $\alpha(\Gamma_{G}) \leq \frac{|G|}{1 - \frac{n-1}{-1}} = \frac{|G|}{n}$. Hence, $G$ would have the EKR property.
	
\begin{lem}\label{lem:bestwecan}
	Let $G\leq \sym(n)$ be a transitive group. If there is a weighing of the conjugacy classes of derangements of $G$ so that the non-trivial representations in the permutation character give the least eigenvalue in LP~\eqref{general-linear-program}, then the maximum  LP~\eqref{general-linear-program} can give is $n-1$. 
\end{lem}
\begin{proof}Assume that 
\[
\fix(g) = \sum_{j=0}^\ell m_j \chi_j(g)
\]
where $\chi_0$ is the trivial representation and $\chi_1,\chi_2,\ldots,\chi_\ell$ are the other constituents of the permutation character of $G$. Note $n=\fix(id) \geq \sum_{j=0}^\ell m_j \chi_j(id)$, where $m_j$ is the multiplicity of $\chi_j$.

Assume that the conjugacy classes $D_i$, for $i =1,\dots, k$, are the conjugacy classes of derangements of $G$. Let $g_1,g_2,\ldots, g_k$ be representatives of $D_1,D_2,\ldots, D_k$, respectively.
Let $(\omega_i)_{i=1,\ldots, k}$ be a weighting on the conjugacy classes such that 
\[
-1 \leq \lambda_{\chi_j} = \frac{1}{\chi(id)} \sum_{i =1}^{k} \omega_i |D_i | \chi_j(g_i), 
\]
for every irreducible character $\chi_j$ of $G$. In particular, for any $j \in \{1,2,\ldots,\ell\}$, we write 
\[
-( \chi_j(id) ) \leq \sum_{i =1}^{k} \omega_i |D_i | \chi_j(g_i)  .
\]
Summing over all the $\chi_j \neq \chi_0$ (with their multiplicities) in the decomposition of $\fix(g)$, we get
\[
-(n-1) = -\sum_{j=1}^\ell m_j (\chi_j(id)) 
\leq  \sum_{j=1}^\ell \sum_{i =1}^{k} \omega_i |D_i | m_j \chi_j(g_i)  
=  \sum_{i =1}^{k} \omega_i |D_i | (-1).
\]
Therefore $ \sum_{i =1}^{k} \omega_i |D_i | \leq n-1$. Consequently, the maximum value of $\lambda_{\chi_0}$ for any such weighting is $n-1$. 
\end{proof}

\section{The General Linear group $\gl{q}$}\label{sect:GL}

In this section $q$ is assumed to be a prime power. The general linear group $\gl{q}$ acts naturally on the non-zero vectors of $\mathbb{F}_q^2$ by left multiplication. The size of the maximum intersecting sets in $\gl{q}$ have been determined in~\cite{MR3188504}, and there are two intersecting families of maximum size. The first family is the collection of canonical intersecting sets (i.e., cosets of point-stablizers in $\gl{q}$). The other
family is the collection of all subgroups that are the stabilizer of a line and their cosets; these are the subgroups
\[
H_\ell = \{ M \in \gl{q} \, | \, \forall v \in \mathbb{F}_q^2, \, Mv-v \in
\ell \}
\]
for some line $\ell$ of $\mathbb{F}_q^2$, and their cosets. These two families are believed to be the only maximum intersecting sets~\cite{ahanjidehCommunication}.

We will be using the notation from Adams~\cite{adams2002character}. The action of $\gl{q}$ on the $q^2-1$ 
non-zero vectors in $\mathbb{F}_q^2$ is not 2-transitive, so this permutation group is simply transitive. 
This action has $q$ orbitals, which are described below. 

\begin{enumerate}
\item The diagonal $\left\{(v,v)\mid v\in \mathbb{F}_q^2 \setminus \{0\}\right\}$ is clearly an orbital of size $q^2-1$.
\item There are $q-2$ orbitals each of size  $(q^2-1)$. The representatives of these
orbitals are $(v,cv)$ where $v\in \mathbb{F}_q^2 \setminus \{0\}$ and $c\in\mathbb{F}_q \backslash\{0, 1\}$. 
\item There is one final orbital of size $(q^2-1)(q^2-q)$, a representative of this
orbital is $(u,v)$ where $u$ and $v$ are not co-linear elements of $\mathbb{F}_q^2 \setminus \{0\}$. 
\end{enumerate}

\subsection{Conjugacy classes of $\gl{q}$}
\label{subsec:conjclasses}

Still following Adams'~\cite{adams2002character} notation, the conjugacy classes of $\gl{q}$ can be divided into four categories
denoted by $c_1(x), c_2(x), c_3(x,y)$ and $c_4(z)$. The structure of the matrices in these categories can be used to count the number of derangements in $\gl{q}$.

The first category is the matrices that are similar to the matrix of the form
\[
c_1(x) = \begin{bmatrix} x &0 \\0 &  x\end{bmatrix},
\]
for some $x \in \mathbb{F}_q^*$. These matrices have one eigenvalue and are diagonalizable over $\mathbb{F}_q$.
Each conjugacy class in
this category has size 1. The conjugacy class that contains the identity is $c_1(1)$. If $x \neq 1$, then $c_1(x)$ is a conjugacy
class of derangements. Thus, there are $q-2$ conjugacy classes of
derangements in this category, each of size 1.

The next category is the set of matrices similar to a matrix of the form
\[
c_2(x) = \begin{bmatrix} x & 1 \\0 &  x\end{bmatrix},
\]
for $x \in \mathbb{F}^*_q$. These matrices have only one eigenvalue and are not diagonalizable. Each conjugacy class in
this category has size $q^2-1$. If $x \neq 1$, then $c_2(x)$ is a conjugacy
class of derangements. Thus, there are $q-2$ conjugacy classes of
derangements in this category, each of size $q^2-1$.

The third category is the set of matrices similar to a matrix of the form
\[
c_3(x,y) = c_3(y,x) = \begin{bmatrix} x &  0 \\0 &  y\end{bmatrix}
\]
for $x, y \in \mathbb{F}^*_q$ and $x\neq y$. These matrices have two 
distinct eigenvalues in $\mathbb{F}_q$. Each conjugacy class in
this category has size $q(q+1)$. If $x, y \neq 1$, then $c_3(x,y)$ is a conjugacy
class of derangements. Thus, there are $\binom{q-2}{2}$ conjugacy classes of
derangements in this category, each of size $q(q+1)$.

The last category corresponds to matrices that do not have eigenvalues in $\fld_q$.
Use $\flde_q$ to represent the unique quadratic extension of $\fld_q$, so these matrices have
eigenvalues in $\flde_q$. If $A\in \gl{q}$ is such matrix, 
then its characteristic polynomial $f(t) = t^2 +at +b$ is irreducible over $\fld_q$. Hence, 
for $q$ odd,  $\Delta = a^2-4b$ is not a square in $\fld_q$. If $\delta^2 = \Delta$, for some 
$\delta \in \flde_q \setminus \mathbb{F}_q$, then we can identify $\flde_q$ 
to be $\mathbb{F}_q(\delta)$. The map 
$z = x+\delta y \mapsto \begin{bmatrix}
	x & \delta y\\
	y & x
\end{bmatrix}$, 
for $x,y\in \mathbb{F}_q$ with $(x,y)\neq (0,0)$, is an embedding of $\flde_q^*$ into $\gl{q}$
(this matrix is called the companion matrix of the element in $\flde_q^*$).
Thus, the final category of conjugacy classes are the matrices similar to a matrix of the form 
\[
c_4(z) = \begin{bmatrix} x &  \delta y \\ y &  x \end{bmatrix},
\]
with $\delta \in \flde_q \setminus \mathbb{F}_q$ where $z=x+\delta y \in \flde_q \setminus \mathbb{F}_q$. 
Each conjugacy class in this category has size $q(q-1)$. All of these
conjugacy classes are classes of derangements. Thus, there are
$\binom{q}{2}$ conjugacy classes of derangements in this category,
each of size $q(q-1)$.

The total number of derangements in $\gl{n}$ is
\[
(q-2)(1) + (q-2)(q^2-1) + \binom{q-2}{2}(q(q+1)) + \binom{q}{2}(
q(q-1)) = q(q^3-2q^2-q+3).
\]

\subsection{Irreducible Representations of $\gl{q}$}
\label{subsec:IrrReps}

Again, we use the notation for the irreducible representations of $\gl{q}$ given in~\cite{adams2002character} where more details can be found. The irreducible representations for $\gl{q}$ are, like the conjugacy classes, split into four categories.

The first category is the set of degree-$1$ representations, denoted by $\rho'(\alpha)$, where $\alpha$ is an irreducible representation 
of $\mathbb{F}^*_q$. If $\alpha$ is the trivial representation of $\mathbb{F}^*_q$, then $\rho'(\alpha)$ is the trivial representation of $\gl{q}$. 
The next category are the degree-$q$ representations denoted by
$\overline{\rho}(\alpha)$, again $\alpha$ is an irreducible representation of $\mathbb{F}^*_q$. The third category are the 
degree-$(q-1)$ representations denoted by $\pi(\chi)$.  Here $\chi$ is an irreducible representation of $\flde_q$ with 
$\chi \neq \overline{\chi}$. The final category are the degree-$(q+1)$ representations
denoted by $\rho(\mu)$, where $\mu$ is an irreducible representation of $\fld_q^* \times \fld_q^*$.
This character is expressed in terms of the norm map, $N : \flde_q^* \rightarrow \fld_q$ with $N(z) = z^{q+1}$.

The values of these characters on the four categories of conjugacy classes is given in Table~\ref{table:chartable}.

\renewcommand{\arraystretch}{1.5} 
\begin{table}
\begin{tabular}{|cc || c | c | c | c |}  \hline 
 &  Number & $q-1$ & $q-1$ & $\binom{q-1}{2}$ & $\binom{q}{2}$ \\  \hline 
 &  Size &  1 & $q^2-1$ & $q(q+1)$ & $q(q-1)$ \\  \hline 
 &      & $c_1(x)$ & $c_2(x)$ & $c_3(x,y)$ & $c_4(z)$ \\  \hline 
Rep : Dim & number & & & & \\ \hline \hline 
$\rho'(\alpha) : 1$ &  &  &    &   &   \\    \hline 
$\alpha = 1 $   & $1$ & $1$ &  $1$  &  $1$ &  $1$ \\    \hline 
 $\alpha \neq 1$   & $q-2$ & $\alpha(x^2)$ &  $\alpha(x^2)$ &  $\alpha(xy)$ &  $\alpha(Nz)$ \\     \hline \hline

$\overline{\rho}(\alpha):q$ &  & &  &  &  \\  \hline       
$\alpha =1 $ & $1$ & $q $ & 0 & $1$ & $-1$ \\  \hline 
$\alpha \neq 1$ & $q-2$ & $q \alpha(x^2)$ & 0 & $\alpha(xy)$ & $-\alpha(Nz)$  \\ \hline  \hline

$\pi(\chi)$: $q-1$ & &  &  &  &  \\  \hline 
$\chi$ & $\binom{q}{2}$ & $(q-1)\chi(x)$ & $-\chi(x)$ & 0 & $-\chi(z) - \chi(\overline{z})$ \\  \hline \hline

$\rho(\mu):q+1 $ &   &  &  & &  \\  \hline
$\mu=(1,\beta)$ & $q-2$ & $(q+1) \beta(x)$ &  $\beta(x)$ & $\mu(g) + \mu^w(g)$& 0 \\  \hline
$\mu=(\alpha,\beta), \, \alpha \neq 1$ & $\binom{q-2}{2}$ & $(q+1)\alpha(x) \beta(x)$ & $\alpha(x) \beta(x)$ & $\mu(g) + \mu^w(g)$& 0\\ \hline  
\end{tabular}
\caption{Character Table for $\gl{q}$ \label{table:chartable}}
\end{table}

From Table~\ref{table:chartable} we can find the permutation character. Define a character $\chi$ by
\begin{equation}\label{eq:permchar} 
\chi := \one + \overline{\rho}(1) + \sum_{\beta} \rho(1,\beta). 
\end{equation}
This character is the sum of the trivial
representation, one representation of dimension $q$, and $q-2$
of the representations with dimension $q+1$.  These degree-1 and degree-$q$ representations both have $\alpha=1$. The $q-2$ degree-$(q+1)$ representations are the representations with $\alpha = 1$ and $\beta \neq 1$. 
The next result proves that the character $\chi$ is the permutation character of $\gl{q}$.

\begin{lem}
The character $\chi$ is the permutation character for the action of $\gl{q}$ on the non-zero vectors of $\mathbb{F}_q^2$.
\label{lem:perm-char}
\end{lem}
\begin{proof}
Recall that the permutation character of the action of $\gl{q}$ on $\mathbb{F}_q^2 \setminus \{0\}$ is given by $\fix(A) = |\left\{ x\in \mathbb{F}_q^2\setminus \{0\} \mid Ax = x \right\}|$, for any $A\in \gl{q}$. We will prove that $\chi$ is equal the permutation character on each conjugacy class of $\gl{q}$.

The value of $\chi$ on the conjugacy classes of type $c_1(x)$ is 
\begin{align}
	1 + q (1) + \sum_{\beta \in \widehat{\fld_q} \backslash \{1\} }  (q+1) (1) \beta(x)\label{eq:permutation-character}
\end{align}
If $x=1$, then this equals $1 + q + (q-2)(q+1) = q^2-1$. For $x=1$, $c_1(x)$ is the conjugacy class containing the identity, so $\chi$ gives the number of fixed points for this conjugacy class. If $x \neq 1$, then $\chi$ is equal to $ 1 + q + (-1)(q+1) = 0$. For $x\neq 1$, the class $c_1(x)$ is a conjugacy class of derangements, so this number is correct.

The value of $\chi$ on $c_2(x)$ is 
\[
1 + 0 + \sum_{\beta  \in  Irr(\fld_q^*) \backslash \{ \mathbf{1} \}}  \beta(x). 
\]
If $x = 1$, this equals $1+ 0 + (q-2) = q-1$; for $x\neq 1$, this is equal to $1 + 0  -1 = 0$.
In both cases, this is the number of fixed points of the elements in the conjugacy classes.

Since $\alpha=1$ in all the irreducible representations in $\chi$, 
the value of $\chi$ on an element in any conjugacy class in category $c_3(x,y)$ can be calculated to be
\[
1 + \one(xy)  + \sum_{\beta  \in  Irr(\fld_q^*) \backslash \{ \mathbf{1} \}} \left(  \beta(y) + \beta(x) \right).
\]
If $x=1$, then $\sum_{\beta  \in  Irr(\fld_q^*) \backslash \{ \mathbf{1} \}} \beta(x) =q-2$; if $x\neq 1$, the sum is $-1$.
So if either $x=1$ or $y=1$, the value of $\chi$ on an element in $c_3(x,y)$ is equal to $1 + 1 + (q-2) + -1 = q-1$. 
Similarly, if $x$ and $y$ are both not equal to 1, then this sum is $1 + 1+ -1 + -1= 0$. In either case, the value of $\chi$ 
on an element of $c_3(x,y)$ is equal to the number of fixed points of the element.

Finally for the conjugacy class $c_4(z)$ the value of this character is 
\[
1 + \left(-\alpha(Nz)\right) + 0 = 1 + (-1) + 0 = 0
\]
(again, with $\alpha=1$) for all $z$, which is correct since this is a conjugacy class of derangements.
\end{proof}

With these values, it is easy to see that the permutation representation is multiplicity-free and to calculate the dimension of the permutation module of $\gl{q}$.

\begin{cor}\label{cor:DimPermModule}
The permutation representation is multiplicity-free and the permutation module has dimension $q^3+q^2-3q-1$.  
\end{cor}
\begin{proof}
The dimension of the permutation module is the sum of the squares of the dimensions of the irreducible representations in the decomposition. By Lemma~\ref{lem:perm-char}, the dimension of the permutation module is
\[
1 + q^2 + (q-2)(q+1)^2  = q^3+q^2-3q-1. 
\]
\end{proof}

\subsection{Cliques in $\Gamma_{\gl{q}}$}
\label{subsec:cliques}

We can easily show that $\gl{q}$ has the EKR property using the well-known clique-coclique bound. 
We will use a stronger version of this result from~\cite[Section 3.7]{godsil2016erdos} to prove that the 
group also has the EKR-module property.

\begin{theorem}
\label{thm:ccl}
	Let $\{A_0,A_1,\dots, A_d\}$ be an association scheme on $v$ vertices and let $X$ be the
	union of some graphs in the association scheme.   Let $E_j$, where $j =0,1,\dots, d$, be the 
	idempotents of the association scheme, with $E_0 = \frac{1}{v} J$.
 	
	If $C$ is a clique and $S$
	is a coclique in $X$, then
	\[
		|C|\,|S| \le v.
	\]
	If equality holds, and $x$ and $y$ are the characteristic 
	vectors of $C$ and $S$ respectively, then
	\[
		x^TE_jx \, y^TE_jy =0 \qquad (i=1,\ldots,d). \qed
	\]
	\end{theorem}

\begin{lem}\label{lem:clique}
There is a subgroup $H$ of $GL(2,q)$ isomorphic to $C_{q^2-1}$. Further,
$H$ contains all the elements of conjugacy classes $c_1(x)$
and exactly two elements from each of the conjugacy classes of type $c_4(z)$. 
\end{lem}
\begin{proof}
A group generated by the companion matrix of a primitive element of $\flde_q$ is cyclic of order $q^2-1$---this is a Singer subgroup. This subgroup contains the unique subgroup of order $q-1$ that consists of the union of all the conjugacy classes of type $c_1(x)$. The remaining elements of this subgroup are from the conjugacy classes of type $c_4(z)$.
\end{proof}

\begin{lem}\label{lem:glEKR}
The graph $\Gamma_{\gl{q}}$ has a clique of size $q^2-1$ and the group $\gl{q}$ has the EKR-property. 
\end{lem}
\begin{proof}
From Lemma~\ref{lem:clique} the $\gl{q}$ has a subgroup $H$ of size $q^2-1$ in 
which all elements, except the identity, are derangements. 
For any distinct $x, y \in H$, it must be that $x^{-1}y \in H$ and is a derangement, so if $x \neq y$, then $x$ and $y$ are not intersecting. This implies that the elements in $H$ form a clique in $\Gamma_{\gl{q}}$. Thus by the Theorem~\ref{thm:ccl} a coclique in $\Gamma_{\gl{q}}$ can be no larger than $\frac{|\gl{q}|}{q^2-1} = q(q-1)$. Since this is the size of the stabilizer of a point, a canonical coclique is an intersecting set of maximum size.
\end{proof}

Equality in the clique-coclique bound, implies a stronger result. If $\psi$ is an irreducible representation of $\gl{q}$, then the $\psi$-\textsl{projection} of a set $S \in \gl{q}$ is the projection of the characteristic vector of $S$ to the $\psi$-module. 
This projection is given by the matrix
\[
E_\psi(g,h) = \frac{\psi(1)}{|\gl{q}|} \psi(hg^{-1}),
\]
and the projection of $S$ to the $\psi$-module is $E_{\psi} v_S$. 
In particular, if $\psi(S) = \sum_{s\in S} \psi(s)$ is not zero, 
then the projection of $S$ to the $\psi$ module is not zero.

\begin{lem}\label{lem:cliquemodules}
The projection of any maximum coclique in $\Gamma_{\gl{q}}$ to the $\mathbb{C}[\gl{q}]$-module corresponding the 
character $\rho(\mu)$, where $\mu = (\overline{\beta}, \beta)$ with $\beta \neq 1$, 
or corresponding to $\overline{\rho}(\alpha)$, with $\alpha^2=1$, but $\alpha \neq 1$, equals $0$. 
\end{lem}
\begin{proof}
The clique-coclique bound in Theorem~\ref{thm:ccl} holds with equality, so for any irreducible representation $\psi$ of $\gl{q}$
	\[
		v_C^T E_\psi v_C  \, v_S^T E_\psi v_S =0
	\]
where $C$ is a maximum clique and $S$ is a maximum coclique. For any irreducible representation $\psi$ with $E_\psi v_C \neq 0$, it must be that $E_\psi v_S =0$. It suffices to show that $\psi(C) = \sum_{c \in C} \psi(c)$ is not zero for $\psi = \rho(\mu)$, with $\mu = (\overline{\beta}, \beta)$ and $\beta \neq 1$, and for $\psi = \overline{\rho}(\alpha)$, with $\alpha^2=1$, but $\alpha \neq 1$,

Let $\rho(\mu)$ be a character with 
$\mu= ( \overline{\beta}, \beta )$ and $\beta \neq 1$. For the group $H$ in Lemma~\ref{lem:clique}
\[
\rho(\mu) (H) = (q+1)(1)(q-1) + 2 (0) \binom{q}{2}   = q^2-1.
\] 

Similarly, let $\overline{\rho}(\alpha)$ be the character of dimension $q$ with $\alpha^2=1$. Then 
\[
\overline{\rho}(\alpha) (H) 
= q(1)(q-1) + 2  \left( (1) \frac{q-1}{2}  + (1)\frac{1}{2} \binom{q-1}{2} +  (-1)\frac{1}{2} \binom{q-1}{2} \right)= q^2-1. 
\]
Since the value of these characters over a maximum clique is non-zero, the
projection of any maximum coclique to these modules must be zero.
\end{proof}

\subsection{Eigenvalues for  $\gl{q}$}
\label{subsec:evalues}

In this section we give a second proof that $\gl{q}$ has the EKR-module property by calculating the eigenvalues 
of the different classes in the conjugacy class association scheme on $\gl{q}$. 

From Subsection~\ref{subsec:conjclasses} we know that the derangements of $\gl{q}$ belong to 
four families of conjugacy classes: $c_1 = c_1(x)$, with $x\neq 1$; $c_2 = c_2(x)$, with $x \neq 1$; $c_3 = c_3(x, y)$, with $x, y$ both not equal to one; and $c_4 = c_4(z)$. 
Define $X_i$ for $i =1, 2, 3, 4$ to be the graph with vertices indexed by elements of $\gl{q}$ and two vertices $g, h$ 
are adjacent if and only if $g h^{-1}$ belongs to a conjugacy class in the family $c_i$. Then 
$\Gamma_{\gl{q}} = \sum_{i=1}^4 X_i$. The graphs $X_i$ all belong to the conjugacy class association scheme for $\gl{q}$, 
so the eigenvalues can be found using Table~\ref{table:chartable}. These eigenvalues are given in Table~\ref{table:evalues-gl}, the rows give the different types of representations, and the columns are the categories of conjugacy classes of derangements.
For each category of conjugacy class, we record the sum of the value of the character over the different conjugacy classes of derangements in the category. From this, we can easily calculate the eigenvalue of the derangement graph of $\gl{q}$; these are given in the final column. 

\begin{table}[t]
\begin{tabular}{|cc|cccc||c|} \hline
 & Category & $c_1(x)$  & $c_2(x)$ & $c_3(x, y)$  & $c_4(z)$ &  \\ 
 & & $x\neq 1$  &  $x\neq 1$  &  $x, y \neq 1$  &  &  \\  \hline 
   & Number  & $q-2$ & $q-2$ & $\binom{q-2}{2}$ & $\binom{q}{2}$ &  \\   \hline
   & size  & $1$ & $q^2-1$ & $q(q+1)$ & $q(q-1)$ &  \\   \hline
Rep : Dim &number &   &  & & & Eigenvalue  \\ \hline \hline

$\rho'(\alpha)$ : $1$ &&&&&& \\ \hline
$\alpha=1$ & 1& $q-2$ & $q-2$ & $\binom{q-2}{2}$ & $\binom{q}{2}$ & $q^4\!-\!2q^3\!-\!q^2\!+3\!q$\\
$\alpha^2=1$ &1& $q-2$ & $q-2$ & $-\frac{q-3}{2}$ & $-\frac{q-1}{2}$ &
  $q$\\
else &$q-3$& $-1$ & $-1$ & $1$ & 0 & $q$\\ \hline \hline
                 
$\overline{\rho}(\alpha)$ : $q$ &&&&&&\\ \hline 
$\alpha=1$ &1& $q(q-2)$ & $0$ & $\binom{q-2}{2}$ & $-\binom{q}{2}$ & $-q^2+q+1$\\
$\alpha^2=1$ &1& $q(q-2)$ & $0$ & $-\frac{q-3}{2}$ & $\frac{q-1}{2}$ &
  $q$\\
else &$q-3$& $-q$ & $0$ & $1$ & 0 &$q$ \\ \hline \hline
        
$\pi(\chi)$ : $q-1$ &&&&&&\\ \hline
$\chi=1$ &$\frac{q-1}{2}$ or $\frac{q}{2}$ & $(q\!-\!1)(q\!-\!2)$ & $-(q-2)$ &
                                                                       $0$ & $q-1$ &$q$ \\
$\chi \neq 1$ &$\frac{(q-1)^2}{2}$ or $\frac{q(q-2)}{2}$ & $-(q-1)$ &
                                                                      $1$
                         & 0 & 0  & $q$ \\ \hline \hline
                         
$\rho(\mu)$ : $q+1$ &&&&&&\\ \hline 
$\alpha=\overline{\beta}$&$\frac{q-3}{2}$ or $\frac{q-2}{2}$ &
                                                               $(q\!+\!1)(q\!-\!2)$
               & $q-2$ & $-(q-3)$ & 0 & $q$ \\
$\alpha=1$ &$q-2$& $-(q+1)$ & $-1$ & $-(q-3)$ & 0 & $-q^2+2q$\\
else & $\frac{(q-3)^2}{2}$ or $\frac{(q-2)(q-4)}{2}$ & $-(q+1)$ &
                                                                       $-1$
                         & $2$ & 0 & $q$ \\ \hline 

\end{tabular}
\caption{The eigenvalues for the conjugacy classes of $\gl{q}$. \label{table:evalues-gl}}
\end{table}

The spectrum of the derangement graph is (the raised number is the multiplicity)
\[
q(q^3-2q^2-q+3)^{(1)}, \quad  q^{(q^4-2q^3-2q^2+4q+1)}, \quad  -q^2+2q^{((q+1)^2(q-2))}, \quad  -q^2+q+1 ^{(q^2)}.
\]
Note that the ratio bound gives a bound of 
\[
\frac{(q+1)q(q-1)^2}{1 - \frac{q(q^3-2q^2-q+3)}{-q^2+q+1}} = \frac{q(q^2-q-1)}{(q-1)},
\]
which does not hold with equality. We next show that there is a
weighted adjacency matrix for which the ratio bound holds with
equality. To get these weights, we set the eigenvalues arising from
non-trivial representations in the permutation representation to be equal to $-1$.

We weight the conjugacy classes of $\gl{q}$ with the weights in Table~\ref{table:glweights}, and 
the eigenvalues of the weighted adjacency matrix are given in Table~\ref{Tab:weightedEvalues}.

\begin{table}[H]
	\begin{tabular}{|c||c|c|c|c|} \hline
Type  & $c_1(x)$, $x\neq 1$   & $c_2(x)$, $x\neq 1$  & $c_3(x,y)$, $x,y \neq 1$ & $c_4(z)$ \\ \hline
Weight & $-\frac{q-1}{q(q-2)}$ &   $\frac{1}{q(q-2)}$ &  $\frac{1}{q(q-3)}$ &  $\frac{1}{q(q-1)}$ \\ \hline
\end{tabular}
\caption{A weighting for the conjugacy classes of derangements in $\gl{q}$}\label{table:glweights}
\end{table}

\begin{table}
\begin{tabular}{|c|c||c|} \hline
 Rep:Dim &number & weighted eigenvalue \\ \hline \hline
 
 $\rho'(\alpha)$: $1$ && \\ \hline
$\alpha=1$ & 1& $q^2-2$ \\
$\alpha^2=1$ (if $q$ is odd) &1&-1\\
else &$q-3$ & $\frac{q-1}{q-2} + \frac{q+1}{q-3}$ \\ 
\hline \hline

$\pi(\chi)$: $q-1$ & &\\ \hline
$\chi=1$ &$\frac{q-1}{2}$ or $\frac{q}{2}$ & $q-3$\\
$\chi \neq 1$ &$\frac{(q-1)^2}{2}$ or $\frac{q(q-2)}{2}$ & $\frac{2}{q-2}$\\ 
\hline \hline

$\overline{\rho}(\alpha)$: $q$ && \\ \hline 
$\alpha=1$ & $1$  &-1\\
$\alpha^2=1$ & $1$& -1\\
else &$q-3$&  $\frac1q \left(\frac{q-1}{q-2} + \frac{q+1}{q-3} \right)$\\ 
\hline \hline

$\rho(\mu)$: $q+1$ && \\ \hline 
$\alpha=\overline{\beta}$&$\frac{q-3}{2}$ or $\frac{q-2}{2}$  & -1\\
$\alpha=1$ & $q-2$&  -1\\
else &$\frac{(q-3)^2}{2}$ or $\frac{(q-2)(q-4)}{2}$& $\frac{2}{q-3}$\\ \hline

\end{tabular}
\caption{Eigenvalues of the weighted adjacency graph for $\gl{q}$. \label{Tab:weightedEvalues}}
\end{table}

The ratio bound on this weighted adjacency matrix gives
\[
\alpha(\Gamma_{\gl{q}}) \leq \frac{|\gl{q}|}{1 -\frac{q^2-2}{-1}} =  q(q-1).
\]
This shows again that $\gl{q}$ has the EKR property.

Recall that the characters that sum up to the permutation character are the $q-2$ representations of dimension 
$q+1$ with $\alpha=1$, the character of dimension $q$ with $\alpha =1$ and the trivial representation. 
All the non-trivial representations have eigenvalue equal to $-1$ under this weighting. There are two other representation that also give the eigenvalue $-1$; $\rho'(\alpha)$  and $\overline{\rho}(\alpha)$ both with $\alpha^2=1$. To show that $\gl{q}$ has the EKR module property, we will need to show that the projection of any maximum coclique to these two modules is equal to 0. Lemma~\ref{lem:cliquemodules} implies this result for the representation $\overline{\rho}(\alpha)$, with $\alpha^2=1$. The representation $\rho'(\alpha)$ with $\alpha^2=1$, is a constituent of the representation induced from the trivial representation on $\slg{q}$, so we need to consider the subgroup $\slg{q}$.

\subsection{The group $\slg{q}$}

The subgroup $\slg{q}$ of $\gl{q}$ also acts transitively on the non-zero vectors of $\mathbb{F}_q^2$. 
The conjugacy classes of derangements are essentially the same, but only the classes where the determinants 
of the matrices are equal to 1 are included in $\slg{q}$. Many of the irreducible characters of $\slg{q}$ are similar 
to the characters of $\gl{q}$. The character table of $\slg{q}$ is given in~\cite{adams2002character}. Using this table 
it is possible to calculate the sum of the value of all the irreducible character over the different conjugacy classes of 
derangements in the different categories.  The tables are slightly different for different values of $q$. We report the values for 
$q \equiv 1\pmod{4}$ and $q \equiv 3 \pmod{4}$ first, and then we discuss when $q$ is even.

For $q$ odd, the eigenvalues for the different conjugacy classes are recorded in
Tables~\ref{Tab:evalues5-1mod4} and~\ref{Tab:evalues-3mod4}.  Like the group $\gl{q}$, the ratio bound does not hold 
with equality for the group $\slg{q}$, so a weighted adjacency matrix must be used. The weightings are given in Table~\ref{table:slweights}.

\begin{table}[H]
	\begin{tabular}{|c||c|c|c|c|} \hline
Type  & $c_1(x)$, $x\neq 1$   & $c_2(x)$  & $c_3(x,y)$, $x,y \neq 1$ & $c_4(z)$ \\ \hline
Weight & $0$ &   $\frac{1}{q-1}$ &  $\frac{1}{q}$ &  $\frac{q^2-3}{q(q-1)^2}$ \\ \hline
\end{tabular}
\caption{A weighting for the conjugacy classes of derangements in $\slg{q}$ with $q$ odd. \label{table:slweights} }
\end{table}

The eigenvalues of the resulting weighted adjacency matrices are given in the final columns of the  Table~\ref{Tab:evalues5-1mod4} and Table~\ref{Tab:evalues-3mod4}.

\begin{table}
\begin{tabular}{|cc|cccc||c|} \hline
 &     & $c_1(x)$  & $c_2(x)$ & $c_3(x, y)$  & $c_4(z)$ & \\ \hline
   & size  & $1$ & $(q^2-1)/2$ & $q(q+1)$ & $q(q-1)$  &  \\  
& &   &  & & & Eigenvalue \\ 
  Rep : Dim &number       &   &  & & & of weighted matrix \\ \hline \hline

$\rho'(\alpha)$ : $1$ & & & & & & \\ \hline
$\alpha=1$ & $1$ & $1$ & $2$ & $\frac{q-3}{2}$ & $\frac{q-1}{2}$ & $q^2-2$ \\ \hline \hline

$\overline{\rho}(\alpha)$ : $q$ & & & & &  & \\ \hline 
$\alpha=1$ &  1 & $q$ & $0$ & $\frac{q-3}{2}$ & $-\frac{q-1}{2}$ & -1 \\ \hline \hline

$\rho(\alpha)$ : $ q+1$ & & & & &  & \\ \hline 
$\alpha(-1)=-1$ &$\frac{q-1}{4}$ & $-(q+1)$ & $-2$ & 0  & 0  & -1\\
else                   & $\frac{q-5}{4}$ & $(q+1)$  &  $2$ & $-2$  & $0 $ & -1  \\ \hline \hline

$\pi(\chi)$ : $q-1$ & & & & &  & \\ \hline
$\chi(-1) =-1$    & $\frac{q-1}{4}$ & $-(q-1)$ & $2$ &   $0$ & $0$  & $\frac{q+1}{q-1}$ \\
$\chi(-1) = 1$    & $\frac{q-1}{4}$ & $ (q-1)$ & $-2$ &   $0$ & $2$ & $2\frac{q^2-5}{(q-1)^2}$ \\ \hline \hline 

$\pi(\chi)$ : $\frac{q+1}{2}$ & & & & & & \\ 
$w_e^{\pm}$ & 2&  $\frac{q+1}{2}$ & 1 & -1 & 0 & -1  \\ \hline  \hline 

$\pi(\chi)$ : $\frac{q-1}{2}$ & & & & & & \\ 
$w_0^{\pm}$ & $2$ &  $-\frac{q-1}{2}$ & 1 & 0 & 0 & $\frac{q+1}{q-1}$ \\ \hline
\end{tabular}
\caption{The eigenvalues for the conjugacy classes of $\slg{q}$ for $q \equiv 1 \pmod 4$. \label{Tab:evalues5-1mod4}}
\end{table}

\begin{table}
\begin{tabular}{|cc|cccc||c|} \hline
   & size  & $1$ & $\frac{q^2-1}{2}$ & $q(q+1)$ & $q(q-1)$  & \\   \hline
 &     & $c_1(x)$  & $c_2(x)$ & $c_3(x, y)$  & $c_4(z)$ & \\ \hline
      &  &   &  & & & Eigenvalue \\ 
 Rep : Dim &number    &   &  & & & of weighted matrix \\ \hline \hline

$\rho'(\alpha)$ : $1$ & & & & & & \\ \hline
$\alpha=1$ &$ 1 $& $1$ & $2 $ & $\frac{q-3}{2}$ & $\frac{q-1}{2}$ & $q^2-2$\\ \hline \hline

$\overline{\rho}(\alpha)$ : $q$ & & & &  & &\\ \hline 
$\alpha=1$ & 1 & $q$ & $0$ & $\frac{q-3}{2}$ & $-\frac{q-1}{2}$ & -1 \\ \hline \hline

$\rho(\alpha)$ : $q+1$ & & & & & & \\ \hline 
$\alpha(-1)=-1$ &$\frac{q-3}{4}$ & $-(q+1)$ & $-2$ & 0  & 0  & -1\\
else &                  $\frac{q-3}{4}$ & $(q+1)$  &  $2$ & $-2$  & $0 $  & -1\\ \hline \hline

$\pi(\chi)$ : $q-1$ &&&&&& \\ \hline
        & $\frac{q+1}{4}$ & $-(q-1)$ & $2$ &   $0$ & $0$ & $\frac{q+1}{q-1}$ \\
       & $\frac{q-3}{4}$ & $ (q-1)$ & $-2$ &   $0$ & $2$ & $2\frac{q^2-5}{(q-1)^2}$ \\ \hline \hline

$\pi(\chi)$ : $\frac{q+1}{2}$ &&&&&&\\ \hline
$w_e^{\pm}$ & 2&  $- \frac{q+1}{2}$  & -1 & 0 & 0 & -1 \\ \hline \hline 

$\pi(\chi)$ : $\frac{q-1}{2}$ & & & & & & \\ \hline
$w_0^{\pm}$ & $2$ &   $\frac{q-1}{2}$ & -1 & 0 & 1 & $\frac{q^2-5}{4}$ \\ \hline 
\end{tabular}
\caption{The eigenvalues for the conjugacy classes of $\slg{q}$ for $q \equiv 3 \pmod{4}$. \label{Tab:evalues-3mod4}}
\end{table}

For $q$ even, all the conjugacy classes of derangements are in either category $c_3(x,y)$ or $c_4(z)$.
Again we use the table of the irreducible characters is given in~\cite{adams2002character}. 
Table~\ref{Tab:evalues-even} records the eigenvalues of the different conjugacy classes of derangements. 
The ratio bound does not hold for the adjacency matrix, so a weighted adjacency matrix is used; these weights are recorded in Table~\ref{table:slweightseven}. The final column of Table~\ref{Tab:evalues-even} contains the eigenvalues of the 
weighted adjacency matrix.

\begin{table}[H]
	\begin{tabular}{|c|| c|c|} \hline
Type  &    $c_3(x,y)$, $x,y \neq 1$ & $c_4(z)$ \\ \hline
Weight &   $\frac{1}{q}$ &  $\frac{q+2}{q^2}$  \\ \hline
\end{tabular}
\caption{A weighting for the conjugacy classes of derangements in $\slg{q}$ with $q$ even.}\label{table:slweightseven}
\end{table}

\begin{table}
\begin{tabular}{|cc|cc||c|} \hline
   & size  & $q(q+1)$ & $q(q-1)$  & \\   \hline
 &      & $c_3(x, y)$  & $c_4(z)$ & \\ \hline
           &  &   & & Eigenvalues \\ 
   Rep : Dim &number   &   & &of weighted matrix\\ \hline \hline

$\rho'(\alpha)$ : $1$ & & & & \\ \hline
$\alpha =1$ &$ 1 $&  $\frac{q-2}{2}$ & $\frac{q}{2}$ & $q^2-2$ \\ \hline \hline

$\pi(\chi)$ : $q-1$ &&&& \\ \hline
  $\chi$     & $\frac{q}{2}$ &   $0$ & $1$ &   $\frac{q+2}{q}$\\ \hline \hline

$\overline{\rho}(\alpha)$ : $q$ & & & & \\ \hline 
$\alpha=1$ & 1 & $\frac{q-2}{2}$ & $-\frac{q}{2}$ & -1 \\ \hline \hline

$\rho(\alpha)$ : $q+1$ & & &   &\\ \hline 
  $\alpha$             & $\frac{q-2}{2}$  & $-1$ & 0 & -1 \\ \hline  

\end{tabular}
\caption{The eigenvalues for the conjugacy classes of $\slg{q}$ for $q$ even. \label{Tab:evalues-even}}
\end{table}

The decomposition of the permutation representation of $\slg{q}$ is similar 
to the permutation representation of $\gl{q}$---we omit the proof as it is very similar to the proof for $\gl{q}$.
For $q$ odd it is the following
\begin{equation*}
\chi = \one + \overline{\rho}(1) + 2 \sum_{\alpha} \rho(\alpha) + \pi(w_e^{\pm}),
\end{equation*}
and for $q$ even it is 
\begin{equation*}
\chi = \one + \overline{\rho}(1) + 2 \sum_{\alpha} \rho(\alpha).
\end{equation*}

\begin{lem}\label{lem:sl}
For all $q$ the group $\slg{q}$ has the EKR property. Further, if $S$ is a maximum coclique in $\Gamma_{\slg{q}}$, then the characteristic vector of $S$ is in the permutation module.
\end{lem}
\begin{proof}
For any value of $q$, the ratio between the largest eigenvalue and the least is $-(q^2-2)$ in the weighted adjacency matrix. So for all of these weighted adjacency matrices, the ratio bound gives
\[
\alpha(\Gamma) \leq \frac{|SL(2,q)|}{q^2-1} = q,
\]
which is exactly the order of the stabilizer of a point. Thus, the ratio bound holds with equality for $\slg{q}$ for all $q$, so we conclude that $\sl{q}$ has the EKR property. 

The ratio bound further implies if $v_S$ is the characteristic vector of $S$, then $v_S -\frac{1}{q^2-1} \one$ is a $-1$-eigenvector. 
For all values of $q$, the only representations that afford an eigenvalue of $-1$ are representations in the permutation representation. 
This implies that $v_S$ is in the permutation module.
\end{proof}

\subsection{$\gl{q}$ has the EKR-module property}

In this next section we will prove that $\gl{q}$ has the EKR-module property.

\begin{thm}\label{thm:EKRM}
Let $S$ be a maximum coclique in $\Gamma_{\gl{q}}$. Then the characteristic vector of $S$ is in the permutation module.
\end{thm}

\begin{proof}
The modules with eigenvalue $-1$ in the weighted adjacency matrix correspond to the representations:
\begin{enumerate}
\item $\rho(\mu)$ with $\mu =(\overline{\beta}, \beta)$, or $\mu=(1, \beta)$
\item $\overline{\rho}(\alpha)$ with $\alpha = 1$, or $\alpha^2=1$
\item $\rho'(\alpha)$ with $\alpha^2=1$.
\end{enumerate}
By the ratio bound, the characteristic vector of any maximum coclique 
lies in the span of these modules. The modules in the permutation representation 
are $ \one = \rho'(1)$, $\overline{\rho}(1)$ and all $\rho(\mu)$ 
with $\mu=(1,\beta)$. To prove this theorem is it necessary to show that the projection of a maximum coclique to any of the modules with eigenvalue $-1$, that are not in the decomposition of the permutation representation, is 0. 
 By Lemma~\ref{lem:cliquemodules}, the characteristic vector of a maximum coclique cannot be in any $\rho(\mu)$ module with $\mu = (\overline{\beta}, \beta)$ where $\beta \neq 1$, or in any the $\overline{\rho}(\alpha)$ modules with $\alpha \neq 1$.

The last case to be considered is the degree 1 representation $\rho'(\alpha)$ with $\alpha^2=1$, and $\alpha\neq 1$. 
The sum of  all the degree 1 representations of $\gl{q}$ is the representation induced from the trivial representation 
on $\slg{q}$. If $T$ is a transversal for the cosets of $\slg{q}$ in $\gl{q}$, then for $\alpha \neq 1$, $\sum_{x \in T} \rho'(\alpha)(x) = 0$.

Let $S$ be a maximum coclique in $\gl{q}$, by Lemma~\ref{lem:glEKR}, $|S| = q(q-1)$. 
Then $S \cap \slg{q}$ is a coclique of $\slg{q}$, and by Lemma~\ref{lem:sl}, 
it cannot be larger than $q$. Further, for any coset $x\slg{q}$ it must be that $x^{-1} ( S \cap x \slg{q})$ is also a 
coclique in $\slg{q}$, and so $|S \cap x \slg{q}| \leq q$. Since the sets $S \cap x \slg{q}$ partition $S$ and $|S| = q(q-1)$, 
each $|S \cap x \slg{q}|$ has size exactly $q$.

For any $\rho'(\alpha)$,
\[
\rho'(\alpha)(S)  = \sum_{x \in T} \rho'(\alpha) (S \cap x \slg{q})= q \sum_{x \in T}  \rho'(\alpha)(x)
\]
which equals 0, unless $\alpha = 1$.
\end{proof}

To prove that $\gl{q}$ has the EKR module property, we will prove that the characteristic vectors of the 
canonical intersecting sets form a spanning set for the permutation module. By the Lemma~\ref{thm:EKRM}, we have that the characteristic vector of any canonical intersecting set is in the permutation module. So we only need to show that the span of the these vectors has the same dimension at the permutation module. 

For $x, y \in \mathbb{F}_q^2$, define $v_{(x, y)}$ be the length-$|\gl{q}|$ vector indexed by the elements in $\gl{q}$. 
The $g$-entry of $v_{(x, y)}$ is $1$ if $g(x) = y$, and $0$ otherwise---these are the characteristic vectors of the canonical cocliques.  Next pick a set of pairwise non-colinear 
vectors $\{ x_i \, : \, i=1,2, \dots, q+1\}$ from $\mathbb{F}_q^2 \setminus \{0\}$. For each $x_i$ with $i=1,2,\dots, q+1$, define
\[
S_i = \{ v_{(x_i, y )}  \, | \, y \in \mathbb{F}_q^2 \setminus \{0\} \} .
\]
This means that each $S_i$ is a set of $q^2-1$ vectors, and there are $q+1$ such sets.

\begin{lem}
The set $S_1 \cup S_2 \cup \dots \cup S_{q+1}$
is a spanning set for the permutation module of $\gl{q}$.
\end{lem}
\begin{proof}   
Each canonical coclique is a maximum coclique and by Theorem~\ref{thm:EKRM} the vectors $v_{(x,y)}$ are in permutation module. It only remains to show that the span of these vectors is the entire module. From Corollary~\ref{cor:DimPermModule}, it is sufficient to show that span of these vectors has dimension $q^3+q^2-3q-1$.  

Define a matrix $N$ with columns the characteristic vectors in the sets $S_i$, for $i \in \{1,\dots, q+1\}$. Order these vectors so that the vectors within a single set $S_i$ are consecutive, and, within $S_i$, the vectors $v_{(i,y_1)}$ and $v_{(i,y_2)}$ with $y_1$ and $y_2$ co-linear are consecutive. 

The dot product of any $v_{(i,j)}$ and $v_{(i,k)}$ is $q(q-1)$ if $j=k$, and 0 otherwise. The dot product of any two vectors $v_{(i,j)}$ and $v_{(a,b)}$ with $i \neq a$, and $j$ not co-linear with $b$ is equal to 1. Then $N^TN$ is
\[
q(q-1)I_{(q+1)(q^2-1)} + 
\big(  \left( J_{q+1} - I_{q+1} \right) \otimes \left( (J_{q+1} - I_{q+1}) \otimes J_{q-1} \right) \big)
\]
This is a square matrix with $(q+1)(q^2-1)$ rows and columns.
The spectrum is
\[
\{  q(q^2-1)^{(1)}, \, (q^2-1)^{(q^2)}, \, q(q-1)^{((q-2)(q+1)^2)}, \, 0^{(2q)} \}
\]
(the numbers in parentheses above the numbers is the multiplicity of the eigenvalue).
Thus the rank of $NN^T$, and hence $N$, is $(q+1)(q^2-1) -2q = q^3+q^2-3q-1$, as required.
\end{proof}

Since the characteristic vector of any maximum intersecting set in $\gl{q}$ is in the permutation module and 
can be expressed as a linear combination of the canonical cocliques. So we conclude that $\gl{q}$ has the EKR module property.
We will prove the same result of $\slg{q}$, using a slightly different approach.

\begin{lem}
The group $\slg{q}$ has the EKR module property.
\end{lem}
\begin{proof}
Define the matrix $N$ so that the rows correspond to the elements in $\slg{q}$ and the columns pairs of elements from $\fld_q^2$. The $(g, (i,j) )$ entry of $N$ is $1$ if $i^g = j$ and $0$ otherwise. The columns of this matrix are the characteristic vectors of the canonical cocliques, by Lemma~\ref{lem:sl} these are in the span of the $q^2-2$- and $-1$-eigenspace. So it remains to prove that the rank of this matrix is one more than the dimension of the -1-eigenspace of the weighted adjacency matrix.

Consider the matrix $NN^T$. The $(g,h)$ entry of this matrix is the number elements on which $g$ and $h$ agree. If $g=h$ the entry is $q^2-1$. 
The non-deragements in $\slg{q}$ belong to the two conjugacy classes: $c_2( 1, 1)$ and  $c_2( 1, \gamma)$.
If $gh^{-1}$ is in the conjugacy class $c_2( 1, 1)$ or $c_2(1,\gamma)$, then the $(g,h)$-entry is equal to $q-1$.
All other entries of $NN^T$ are equal to 0.

This means that $NN^T$ is equal to 
\[
(q^2-1) I_{q(q^2-1)} + (q-1) ( A_1 + A_2), 
\]
where $A_1$ and $A_2$ are the adjacency matrices in the conjugacy class association scheme corresponding to the conjugacy classes $c_2( 1, 1)$ and $c_2( 1, \gamma)$. 

The eigenvalue of $A_1+A_2$ can be calculated using the character table of $\slg{q}$ and are given in Table~\ref{tab:sl A1A2}.

\begin{table}
\begin{tabular}{|l|ccccccc|} \hline
Rep. & $\rho(\alpha)$ & $\overline{\rho}(1)$ & $\rho'(1)$ &  $\pi(\chi)$ & $\omega_e^\pm$ & $\omega_0^\pm$ & $\omega^\pm$ \\ \hline
Eigenvalue & $q-1$  &    0   &  $q^2-1$ & $-(q+1)$ & $q-1$ & $-(q+1)$ & 0\\ 
Multiplicity& $(q+1)^2\frac{q-3}{2}$ & $q^2$ & $1$ & $\frac{(q-1)^3}{2}$ & $2(\frac{q+1}{2})^2$ &  $2(\frac{q-1}{2})^2$ & $2q^2$ \\ \hline
\end{tabular}
\caption{Eigevnalues of $A_1+A_2$. \label{tab:sl A1A2}}
\end{table}

From this it can be seen that eigenvalues of $NN^T$ are
\[
\left( 
( (q^2-1)+(q-1)^2 )^{\left(  \frac{(q-3)(q+1)^2}{2} \right)}, \quad 
(q^2-1)^{(2q^2)}, \quad
q (q^2-1)^{(1)}, \quad 
0^{\left( \frac{(q-1)^3}{2} + 2(\frac{(q-1)}{2})^2 \right)}
\right).
\]
The rank of $N$ is  $(q-1)q(q+1)-\frac{q(q-1)^2}{2} = \frac{q(q-1)(q+3)}{2}$, which is one less than the dimension of the -1-eigenspace of the weighted adjacency matrix for $\slg{q}$.

For $q$ even there is only one conjugacy class of non-derangements. 
So $NN^T$ is equal to 
\[
(q^2-1) I_{q(q^2-1)} + (q-1) ( A_1), 
\]
where $A_1$ is the adjacency matrix in the conjugacy class scheme that corresponds to the single class of non-derangements.
The eigenvalues of $A_1$ are 
\[
\left( q^2-1^{(1)}, (q-1)^{ \left( \frac{(q+1)^2(q-2)}{2} \right) }, 0^{(q^2)}, -(q+1)^{\left( \frac{q(q-1)^2}{2} \right)} \right).
\]

We deduce that eigenvalues of $NN^T$, for $q$ even, are
\[
\left( 
q (q^2-1)^{(1)}, \quad 
((q^2-1)+(q-1)^2)^{\left( \frac{(q+1)^2(q-2)}{2} \right)}, \quad 
(q^2-1)^{(2q^2)}, \quad
0^{\left( \frac{q(q-1)^2}{2} \right)}
\right).
\]
The rank of $N$ is  $(q-1)q(q+1)-\frac{q(q-1)^2}{2} = \frac{q(q-1)(q+3)}{2}$, which is one less than the dimension of the -1-eigenspace of the weighted adjacency matrix for $\slg{q}$.

\end{proof}

\begin{thm}
The group $\gl{q}$ does not have the strict-EKR property.
\end{thm}
\begin{proof}
For a line $\ell$, let $S_\ell$ be the set of all $M \in \gl{q}$
with $Mv-v \in \ell$ for all $v \in \mathbb{F}_q^2$. 
This forms a group of size $q(q-1)$, this can be seen by counting the number of
matrices in $S_\ell$. Assume without loss of generality that $\ell$ is the line containing $(0,1) \in \mathbb{F}_q^2$. Then any matrix in $S_\ell$ has the $(1,1)$-position equal to $1$ and the $(1,2)$-position equal to $0$, then there are $(q-1)$ choices for the
$(2,2)$-entry, since it cannot be $0$, and $q$ choices for the $(2,1)$-entry.
Finally, from the structure of these matrices, it can be seen that each matrix in $S_\ell$ has a fixed point. 
Since $S_\ell$ is a subgroup, for any $M_1, M_2 \in S_\ell$ the matrix $M_1M_2^{-1}$ is also in $S_\ell$ so it has a fixed point. This shows that $S_\ell$ 
is an intersecting set.
\end{proof}

\section{$\agl{q}$ on lines}

In this section we will examine two related imprimitive groups that do not have the EKR property and for which the method used in Section~\ref{sect:GL}
does not seem to produce good bounds. The first group is the affine general linear group, $\agl{q}$, with the action on lines, rather than points. 
This action is related to $\pgl{q}$ acting on pairs of projective points; this is the second group that we consider.

Recall that the affine plane $\ag(2,q)$, for any prime power $q$, is the incidence structure $(\mathcal{V}_q,\mathcal{L}_q,\sim)$, 
where the set of points is $\mathcal{V}_q = \mathbb{F}_q^2$, the set of lines is $\mathcal{L}_q = \left\{ L_{u,v} \mid u,v\in \mathbb{F}_q^2, v\neq 0 \right\}$ with $L_{u,v} = \left\{ u+tv \mid t\in \mathbb{F}_q \right\}$, and for any $x\in V$, $\ell \in L$, $x\sim \ell$ if and only if $x\in \ell$. 
The permutation group $\agl{q}$ consists of all affine transformations $(M,z): v\mapsto Mv+z$, for any $M \in \gl{q}$ and $z \in \mathbb{F}_q^2$. 
Hence, $\agl{q}$ acts naturally on the vector space $\mathbb{F}_q^2$, which coincides with the points of $\ag(2,q)$. 
This action is 2-transitive and so, under this action, $\agl{q}$ has both the EKR-property and the EKR-module property. 

The affine group $\agl{q}$ also acts on the set of lines of $\ag(2,q)$ as follows: for any $(M,z) \in \agl{q}, L_{u,v} \in \mathcal{L}_q$, we have 
$(M,z)(L_{u,v}) = \{  (M,z)(u+tv) \mid t\in \mathbb{F}_q \}$. We will refer to this action as the action on the lines and 
this is the action we consider in the section. This action is not 2-transitive, it is a rank 3 imprimitive action. There are
$q+1$ blocks each of size $q$; each block is a set of parallel lines. This means each block has exactly one line 
through 0 and all the other lines are shifts of this line. Since this group is imprimitive of rank  $3$, the system of imprimitivity describe above is the only one.

\subsection{Derangements in $\agl{q}$}

In this section we will find the conjugacy classes of derangements in $\agl{q}$. 

\begin{lem}
If $M$ has no eigenvalues in $\fld_q$, then $(M,z)$ is a derangement for any $z \in \mathbb{F}_q^2$.
\end{lem}
\begin{proof}
Let $(M,z)$ be an element of $\agl{q}$ that fixes the line $\ell =\ell_0 +w$ where $\ell_0$ is the line through zero given by $\langle v \rangle$.
For any $i\in \mathbb{F}_q$, the point $iv+w$ is on $\ell$, so $(M,z) (iv+w)$ is also on $\ell$. Thus
\[
(M,z) (iv+w) = M (iv+w) + z = M(iv) + M(w) +z = M(iv) + (M,z)(w).
\]
Since $(M,z)(w) \in \ell$, the vector $M(iv)$ is the difference of two points both on the line $\ell$. This implies $M(iv)$ is on the line $\langle v \rangle$ 
and $v$ is an eigenvector for $M$.  

Thus, if an element $(M,z)$ is not a derangement, then $M$ has an eigenvector; 
the contrapositive of this statement is that if $M$ has no eigenvalues, then $(M,z)$ is a derangement.
\end{proof}

\begin{lem}
Assume $M$ is not diagonalizable and has exactly one eigenvalue with corresponding eigenvector $s$. 
Then $(M, z)$ is a derangement if and only if the only eigenvalue of $M$ is equal to $1$ and $z \not \in \langle s \rangle$.
\end{lem}
\begin{proof}
First, if $z \in \langle s \rangle$ then $(M, z)$ fixes the line through zero given by $\langle s \rangle$. 
So clearly in this case $(M, z)$ is not a derangement.

Assume that $M$ has only one eigenvector $s$ and the corresponding eigenvalue is $\mu \neq 1$. Then the vector
\[
w = (M-I)^{-1}(s-z)
\]
is defined and $(M, z)$ is not a derangement since it fixes the line $\langle s \rangle + w$. To see this consider for any $k$,
\[
(M, z) (ks+w) 
= \mu k s + (M-I)w +w+z 
= \mu k s + s-z+w+z 
= (\mu k +1)s +w.
\]

Assume $Ms =s$, $z \not \in  \langle s \rangle$ and that $(M, z)$ fixes the line $\ell =\ell_0 +w$ where $\ell_0$ 
is the line through zero given by $\langle v \rangle$. 
Then $(M,z)(w) = kv+w$ for some $k$, and that $Mw-w+z$ is in the line $\ell_0$.
Similarly,  $(M,z)(v+w) = k'v+w$ for some $k'$, so $Mv+Mw-w+z$ is also the line $\ell_0$.
This implies that $Mv$ is on $\ell_0$, so $v$ is an eigenvector. As $M$ has only one eigenvector, $\ell_0 = \langle s \rangle$. 
Further $Mw-w = (M-I)w$, must be in $\langle s \rangle$, since the eigenvalue corresponding to $s$ is 1. But then the fact that 
$Mw-w+z$ is on the line $\ell_0$ implies that $z$ is a multiple of $s$.
\end{proof}

\begin{lem}\label{lem:twodistinct}
If $M$ has two distinct eigenvalues, then $(M, z)$ is not a derangement, for any $z\in \mathbb{F}_q^2$.
\end{lem}
\begin{proof}
Assume that $v_1$ and $v_2$ are eigenvectors of $M$ with corresponding, distinct, eigenvalues 
$\mu_1$ and $\mu_2$. Since the eigenvalues are distinct, we can assume that $\mu_2 \neq 1$.

Set $\ell_0 = \langle v_1 \rangle$ and express $z = a_1 v_1 + a_2 v_2$. We claim that $(M, z)$ fixes the line 
\[
\ell_0 + \frac{-1}{\mu_2-1}z = \ell_0 + \frac{-a_2}{\mu_2-1}v_2.
\] 
To see this, consider:
\begin{align*}
(M, z)(\ell_0 + \frac{-a_2}{\mu_2-1}v_2) 
&= M(\ell_0 + \frac{-a_2}{\mu_2-1}v_2) + ( a_1 v_1 + a_2 v_2) \\
&= \ell_0 + \frac{-a_2}{\mu_2-1} \mu_2 v_2 + ( a_1 v_1 + a_2 v_2)  \\
&= \ell_0 + \frac{-a_2}{\mu_2-1} \mu_2 v_2 +  a_2 v_2  \\
&=\ell_0 + \frac{-a_2}{\mu_2-1} v_2 . 
\end{align*}
\end{proof}

Finally, we consider the case where $M$ is diagonalizable and both eigenvalues are equal, so $M$ is a scalar multiple of that identity matrix.  In this case it is clear that $(M, z)$ fixes the line $\langle z \rangle$. 

\begin{lem}
If $M$ is a scalar multiple of the identity matrix, then $(M, z)$ is not a derangement for any $z$.
\end{lem}

In summary, in $\agl{q}$ there is one conjugacy class of derangements of the form $(M, z)$ where $M$ has 1 as its only eigenvalue and $z$ is not an eigenvector for $M$. This class has size $(q^2-1)(q^2-q)$ and we denote it with $C_0$. There is a family of $\binom{q}{2}$ conjugacy classes each of the form $(M, z)$ where $M$ has no eigenvalues. Each conjugacy class in this family has size $q^3(q-1)$, we will label these conjugacy classes by $C_i$ with $i=1, \dots, \binom{q}{2}$.  Further, the permutations in conjugacy classes $C_i$ with $i=1, \dots, \binom{q}{2}$ fix none of the blocks of imprimitivity of $\agl{q}$.

\subsection{Permutation Representation of $\agl{q}$} 

Many of the irreducible representation of $\agl{q}$ arise from a
representation on $\gl{q}$, of the representations that do not arise from an irreducible 
representation of $\gl{q}$, there are $q-1$ with dimension
$q^2-1$, and one with dimension $(q-1)(q^2-1)$.  Since $\agl{q}$ is a rank 3 imprimitive group, it is straightforward to find the permutation representation of it.

\begin{lem}
Let $G$ be an imprimitive group with rank 3. 
Then, the permutation representation of $G$ is the sum of three irreducible representations: 
the trivial representation, $\chi_1$ and $\chi_2$, where $\chi_1$ is the permutation representation from the 
action of $G$ on the blocks.
\end{lem}
\begin{proof}
Since the group has rank, 3 it is clear that the permutation representation is the sum of 3 distinct irreducible representations, 
one of which must be the trivial character.

Let $\chi$ be the permutation representation of $G$, and $\chi_1$ the permutation representation of $G$ for the action of $G$ on the blocks.  Let $G_1$ denote the stabilizer of a point in $G$ and $G_B$ the stabilizer of a block.
Then 
\[
\langle \chi, \chi_1 \rangle_G
\langle \Ind( 1_{G_1}) ^{G}, \Ind( 1_{G_B} )^{G} \rangle_G \,
=
\langle  1_{G_1}, \Res \left(  ( \Ind( 1_{G_B} )^{G} ) \right)_{G_1} \rangle_{G_1}
\]
This equals the number of orbits $G_1$ has on the blocks, which is 2.
Both representations include the trivial representation, so $\chi$ includes $\chi_1$ with multiplicity 1.
\end{proof}

In particular, $\chi_1$ is the permutation representation from the action of  $\agl{q}$ on the blocks, 
minus the trivial representation, so $\chi_1(g) = \fix_{blocks}(g) - 1$. This is the $q$-dimensional representation arising 
from the representation $\overline{\rho}(1)$ of $\gl{q}$. Further, $\chi_2 = \fix(g) - \fix_{blocks}(g)$ is an irreducible degree $q^2-1$ 
permutation representation of $\agl{q}$ and $\chi_2$ restricted to $\gl{q}$ is the permutation representation on $\gl{q}$.

Since the permutations in the conjugacy classes $C_i$ with $i \in \{1,\dots, \binom{q}{2} \}$ do not fix any of the 
blocks, $\chi_1(x) = -1$ for each $x \in C_i$. Further, $\chi_1(x)=0$ for any $x \in C_0$, since these permutations 
fix exactly one block. 

To apply the method used for $\gl{q}$ and $\slg{q}$, a weighting must be found for the conjugacy classes so that
$\lambda_{\chi_1} \geq -1$ and $\lambda_{\chi_2} \geq -1$ with $\lambda_{\one}$ is maximized. It is possible to give a 
formula for this eigenvalue where the weighting on $C_i$ is denoted by $a_i$:
\[
\lambda_{\chi_1} = \frac{1}{q} \left( a_0 |C_0| 0 + \sum_{i=1}^{\binom{q}{2}} a_i |C_i| (-1) \right)
 =-q^2(q-1) \sum_{i=1}^{\binom{q}{2}} a_i
 \]
 and
 \[
\lambda_{\chi_2} = \frac{1}{q^2-1} \left( a_0 |C_0| (-1) + \sum_{i=1}^{\binom{q}{2}} a_i |C_i| ( 0 )  \right)
 = -a_0 (q^2-q).
\]
It is straight-forward to see that an appropriate weighting will have both
\[
a_0 \leq  \frac{1}{q^2-q},   \qquad \sum_{i=1}^{\binom{q}{2}} a_i \leq \frac{1}{q^2(q-1)}.
\] 
As predicted by Lemma~\ref{lem:bestwecan}, the value of the trivial character is
\[
\lambda_{\chi_1} 
= \left( a_0 |C_0| 1 + \sum_{i=1}^{\binom{q}{2}} a_i |C_i| (1) \right)
\leq \frac{ (q^2-1)(q^2-q) }{q^2-q} + \frac{q^3(q-1)}{q^2(q-1)} 
= q^2-1+ q. 
 \]
The equation in the ratio bound gives
\[
\alpha \leq \frac{(q-1)q^3(q+1)}{1- \frac{q^2+q-1}{-1}} = \frac{(q-1)^2q^3(q+1)}{q^2+q} = (q-1)^2q^2. 
\]
But this is not a bound on the size of a coclique, since we will see in the next section that there is a larger coclique. 
The reason that this does not give a bound is that there will be other irreducible characters with eigenvalue smaller than -1.

\subsection{Intersecting sets in $\agl{q}$}

In this section we prove Theorem~\ref{thm:AGL}. First we will give a weak upper bound on the size of an intersecting set in 
$\agl{q}$. Second, we will show that $\agl{q}$ does not have the EKR property by constructing cocliques in 
$\Gamma_{\agl{q}}$ that are larger than the stabilizer of a point. First we not that there is a subgroup in $\agl{q}$ in which every element except the identity is a derangement.

\begin{lem}
There is a subgroup in $\agl{q}$ of size $q+1$ in which all non-identity elements are derangements.
\end{lem}
\begin{proof}
Such a group is the cycle subgroup generated from any permutation that, when restricted to the blocks is a $(q+1)$-cycle.
\end{proof}

Just as in Lemma~\ref{lem:clique}, this implies that the derangement graph $\Gamma_{\agl{q}}$ has a clique of size $q+1$.
 Theorem~\ref{thm:AGL} follows from this lemma by the clique-coclique bound, Theorem~\ref{thm:ccl}, since if $\mathcal{F} \subset \agl{q}$ is intersecting, then $|\mathcal{F}|\leq \frac{ | \agl{q} | }{q+1} = q^3(q-1)^2$.

Next we construct a set of intersecting permutations from $\agl{q}$ that is larger than the canonical intersecting set. To do this we first need some facts.

\begin{prop}\label{prop:blockespace}
If $(M, z)$ fixes the block $B$ and $\ell = \langle v \rangle$ is the line of $B$ through 0, then $v$ is an eigenvector of $M$.
\end{prop}
\begin{proof}
Assume that $M$ fixes $B$ and $\ell = \langle v \rangle$  is the line of $B$ that includes the zero vector. Since 
$(M, z)$ maps the 0-vector to $z$, we know that $(M, z)$ maps the line $\ell$ to $\ell+z$, in particular,
$(M, z)( v ) = Mv +z$ is on the line $\ell +z$. This means that $Mv$ is on the line $\ell$, so $v$ is an eigenvector of $M$. 
\end{proof}

\begin{lem}
Let $(M, w)$ be an element of  $\agl{q}$. If $(M, w)$ fixes two of the
blocks of imprimitivity, then $(M, w)$ fixes a line.
\end{lem}
\begin{proof}
  Assume that $(M,w)$ fixes the blocks $B_1$ and $B_2$. Let $\ell_1$
  and $\ell_2$  be the lines through the zero vector in $B_1$ and
  $B_2$ (respectively). By Proposition~\ref{prop:blockespace}, $M$ is diagonalizable and $\ell_1$ and $\ell_2$ 
  are eigenspaces of $M$. Let $\mu_1$ and $\mu_2$ be the eigenvalues of $M$ corresponding to
$B_1$ and $B_2$.

\textbf{Case 1.} Assume one of $\mu_1$ and $\mu_2$ is not equal to 1, so without generality we
can assume $\mu_2 \neq 1$. Following the proof of Lemma~\ref{lem:twodistinct}, 
this implies $(M,w)$ fixes the line $\ell_1 + \frac{-1}{\mu_2-1}w$ in $B_1$. 

\textbf{Case 2.} If $\mu_1 = \mu_2 = 1$ then $M$ is the identity
matrix and $(M,w)$ is a shift. Let $\ell$ be the line that contains
the zero vector and $w$. Then $(M,w)$ fixes $\ell$ as well as every other line in
the block that contains $\ell$. In fact, as long as $w$ is not the zero
vector, any such a $(M,w)$ fixes $q$ lines.
\end{proof}

\begin{lem}
The stabilizer of the blocks in $\agl{q}$ is an intersecting set of size $q^2(q-1)$.
\end{lem}
\begin{proof}
  The stabilizer of the blocks consists of all the elements $(M,z)$ of
  $\agl{q}$ where $M$ is a scalar multiple of the identity.  The
  number of such elements is $q^2(q-1)$.  By the previous result, every
  element has a fixed point. Since the stabilizer of the blocks is a group, this implies that it is 
  an intersecting set.
\end{proof}

Let $\stab$ denote the stabilizer of the blocks in $\agl{q}$. Then, $\stab$ and each of its cosets is an 
intersecting set. Next will show that the union of a subset of these cosets forms a larger intersecting set of permutations.
Note that the quotient of $\agl{q}$ with the stabilizer of the blocks is isomorphic to the group $\pgl{q}$. A pair of permutations $(g,h)$ are 
\textsl{2-intersecting} if there are two distinct points $i$ and $j$ so that $h^{-1}g(i) = i$ and  $h^{-1}g(j) = j$; a set of permutations is 2-intersecting if 
any two elements from the set are 2-intersecting.

\begin{lem}
If $S$ is a $2$-intersecting set of permutations in $\pgl{q}$ (with the action on the $q+1$ blocks), 
 then
 \[
 \displaystyle{ \bigcup_{x \in S} \,  x \stab}
 \]
is an intersecting set in $\agl{q}$.
\end{lem}
\begin{proof}
The action of $\pgl{q}$ is the action on the blocks, since $S$ is 2-intersecting for any two elements $x,y \in S$ the permutation $y^{-1}x$ fixes two blocks.
For any two permutations $x\sigma \in x\stab$ and $y\pi \in y\stab$, the permutation $\pi^{-1}y^{-1}x\sigma$ also fixes two blocks and $x\sigma$ and $y\pi$ are intersecting.  
\end{proof}

This motivates finding maximum 2-intersecting sets in $\pgl{q}$. 

\begin{theorem}\label{thm:construction}
If $q$ is odd, then there is a set of 2-intersecting permutations in $\pgl{q}$ with size $(3q-5)/2$.
If $q$ is even, there is a set of 2-intersecting permutations in $\pgl{q}$ with size $(3q-4)/2$.   
\end{theorem}
\begin{proof}

We first consider the case when $q$ is odd, the case for $q$ even is almost identical. We will construct a 2-intersecting set of permutations with size $(3q-5)/2$. This set will contain the identity, as well as $(3q-7)/2$ permutations with two fixed points.
In $\pgl{q}$ any non-identity element has at most 2 fixed points, so no two permutations can agree on more than 2 elements.
When $q$ is odd there are $(q-3)/2$ conjugacy classes, each of size $q(q+1)$, of permutations 
each with exactly two fixed points in which the elements are not involutions; call these the conjugacy classes of Type 1. There is another conjugacy class of size $q(q+1)/2$ with permutations that also have exactly two fixed points and in which the elements are involutions; we will call this the conjugacy class of Type 2.

For any $i \in \{ 1, \dots, q+1\}$ and each conjugacy classes of Type 1 there are exactly $2q$ permutations in the class that fix $i$.
Further, for any $j, k \in \{1, \dots ,q+1\} \backslash [i]$, there are exactly 2 permutations in each conjugacy 
classes of Type 1 that fix $i$ and map $j$ to $k$.

Let $\sigma$ be a permutation in a conjugacy class of Type 1 that fixes $i$. Define the $q$ pairs
\[
\mathcal{P} = \{  (j, k) \, : \,  j^\sigma=k, \quad j, k \in \{1, \dots ,q+1\} \backslash i    \}.
\]
Consider any conjugacy class of Type 1 that does not contain $\sigma$. There are no permutations $\pi$ in this conjugacy class 
that fix $i$ and have two distinct pairs  $(j_1,k_1), (j_2, k_2) \in \mathcal{P}$ such that $j_1^\pi = k_1$ and $j_2^\pi =k_2$.  
Since there are exactly $2q$ permutations in the class that fix $i$, for any pair $(j,k) \in \mathcal{P}$, there are exactly two 
permutations in the conjugacy class that fix $i$ and map $j$ to $k$. Since there are $q$ pairs in $\mathcal{P}$, counting shows that 
each permutation in this conjugacy class that fixes $i$, must also map $j$ to $k$ for exactly one pair $(j,k) \in \mathcal{P}$. In general, this means that for any two permutations from different conjugacy classes of Type 1, if they both fix a common element, then they are actually 2-intersecting. We will use this fact to build a 2-intersecting set of permutations.

For each conjugacy class of Type 1, let $x$ be one of the two permutations that fix both $1$ and $2$. We can assume without loss of generality that $x$ is the diagonal matrix with entries 1 and $a$. Fix $h$ to be the unique permutation that maps $1$ to $2$, $2$ to $3$ and $3$ to $1$; $h$ is the product to 3-cycles and has order 3. Select the set of permutations $\{x, hxh^{-1}, h^{-1}xh \}$. These three permutations all belong to the same conjugacy class and it can easily be seen (by multiplying the matrix representative of the permutations) that they are pairwise 2-intersecting. Further, $x$ fixes the points 1 and 2, $hxh^{-1}$ fixes the points 2 and 3 and $h^{-1}xh$ fixes the points 1 and 3.

If we do this for each conjugacy class of Type 1, we get a set of $3 \frac{q-3}{2}$ permutations. Any two permutations from this set that belong to the same conjugacy class are 2-intersecting. Any two permutations in this set that are from two different conjugacy classes agree on a fixed point (every permutation in this set will fix at least two of the points 1, 2 and 3) so they will also be 2-intersecting.

Finally, the permutations in the conjugacy class of Type 2 are involutions. There is a set of three elements that fix at 
least two of $1, 2$ or $3$; this set forms a triangle in the derangement graph. Using the same argument as for the conjugacy classes of Type 1, if any permutation from a conjugacy class of Type 1 and a permutation from the conjugacy class of Type 2 have a common fixed point, then they are 2-intersecting. So any one of these 3 elements from the conjugacy class of Type 2 can be added to the set. Finally the identity can be added to the set as well to produce a 2-intersecting set of size
\[
3 \frac{q-3}{2}+1 + 1 = \frac{3q-5}{2}.
\]

The same argument work when $q$ is even, but in this case the number of conjugacy classes of Type 1 is $\frac{q-2}{2}$ and there is no conjugacy class of Type 2.
\end{proof}

From these intersecting sets in $\pgl{q}$ it is possible to build intersecting sets in $\agl{q}$.

\begin{theorem}\label{thm:maxinAGL}
For $q$ is odd, there is an intersecting set in $\agl{q}$ with size $\frac{q^2(q-1)(3q-5)}{2} $, and for $q$ even, there is an intersecting set in $\agl{q}$ with size $\frac{q^2(q-1)(3q-4)}{2} $.
\end{theorem}

Since these sets are larger than the stabilizer of the point in $\agl{q}$, this group does not have the EKR property.

\begin{cor}
The group $\agl{q}$ acting on the lines does not have the EKR property.
\end{cor}

From Theorem~\ref{thm:construction} and  Theorem~\ref{thm:AGL} we know that the size of the maximum intersecting sets in $\agl{q}$ is between $\frac{q^2(q-1)(3q-5)}{2} $ and $q^3(q-1)^2$, (for $q$ odd). These two bounds are quite far apart, and we believe that the actual size of the maximum intersecting sets is much smaller than the bound in Theorem~\ref{thm:AGL}.  We conjecture the following.
\begin{conj}
	 If $\mathcal{F} \subset\agl{q}$ is intersecting, then there exist $a\in \mathbb{N}$ and $b\in \mathbb{Z}$ such that 
	 $|\mathcal{F}| \leq q^2(q-1)(aq+b)$.
\end{conj}

The group $\agl{3}$ and $\agl{4}$ are sufficiently small that it is possible to find the size of
the largest intersecting set using Grape~\cite{Soi21}. For $q=3$ the bound can be achieved by 
taking 5 cosets of the stabilizer of the blocks. For $q=4$ the bound can be achieved by the union of four cosets of $\stab$.

\begin{lem}
The size of the largest intersecting set in $\agl{3}$ is 45 and in $\agl{4}$ it is 192.  
\end{lem}

It is straight-forward to solve the linear programming (LP) problem~\eqref{general-linear-program} for $\agl{q}$ 
with small values of $q$. Using Gurobi~\cite{gurobi}, the linear programming bound for $\agl{3}$ give a maximum 
ratio of 5. This solution in the ratio bound proves that intersecting set in $\agl{3}$ is no larger than 72. Similarly, for $q=4$, solving the LP gives a maximum ratio of 9, 
so the best result in the ratio bound is that an intersecting set is no larger than $288$. 
For $\agl{5}$ the best ratio is $9$, giving a bound for an intersecting set of $1200$. This group was too large
for an exhaustive search, but the largest set we were able to find was of size $500$.
For $q=7$ the best ratio is 13, giving a bound of 7056; we were only able to find a set of size $2352$.
It seems that the solution to the (LP) problem~\eqref{general-linear-program} for $q$ odd gives a ratio of $2q-1$, and that this does not give an effective
result in the ratio bound for $q=3,4,5,7$.

\section{Further Work}

In this paper for $\gl{q}$ and $\slg{q}$ only the action on the vectors of $\fld^2_q$ was considered. Questions about the largest 
intersecting sets can be asked for any group action. Since any group action, is equivalent to the action of the group on a set of cosets, this is a relativey tractable problem. 
 Most research on EKR-type results for groups focuses on well-known group actions, there is a growing body of work considering 
 all the actions of a group~\cite{MR3365595, Raghu}. Since the character table $\gl{q}$ is completely understood, it should be straightforward to try 
 this method for different actions of the general linear group. 

It is known that if a group has a regular subgroup, then there is a weighting of the conjugacy classes so that the ratio bound 
holds with equality~\cite[Theorem 3.5]{MahsaMe}. It would be interesting to know of more cases where it can be shown that such a weighting exists.
 
\begin{quest}
If a permutation group $G$ has the EKR property, is there a weighting on the conjugacy classes so that ratio bound holds with equality?
\end{quest}

The group $\agl{q}$ does not have the EKR property, and we were able to construct some larger intersecting set. 
But, we do not know if these sets are the largest possible, in fact we do not even have good bounds on the size of these sets.

\begin{quest}
What are the largest cocliques in the  derangement graphs for $\agl{q}$?
\end{quest}

We conjecture that the maximum intersecting set in $\agl{q}$ will be formed by unions of cosets of the stabilizers of the blocks.
If this is true, then finding the maximum intersecting sets in $\agl{q}$ reduces to finding the maximum 2-intersecting sets 
in $\pgl{q}$, which leads to the next open question.

\begin{quest}
For $q>3$ what is the largest 2-intersecting set in $\pgl{q}$?
\end{quest}

Using Grape~\cite{Soi21} the maximum 2-intersecting sets in $\pgl{q}$ and $\psl{q}$ can be determined for small values of $q$, 
these are recorded in Table~\ref{tab:BestPGL} and Table~\ref{tab:BestPSL}. 
Table~\ref{tab:BestPGL} indicates that the construction in 
Theorem~\ref{thm:construction} is not in general optimal.

\begin{table}
\begin{tabular}{|c|ccc ccc cc|} \hline
$q$  &  3 &4 &5&7&8&9&11&13 \\ \hline  \hline
Size Max. Coclique & 2 &4&5&8&10&12&17 & $\geq 19$\\ \hline
Theorem~\ref{thm:construction} bound &  2 &4&5&8&10&11&14&17\\ \hline
\end{tabular}
\caption{Size of maximum 2-intersecting set in $\pgl{q}$\label{tab:BestPGL}}
\end{table}

\begin{table}
\begin{tabular}{|c|ccc ccc cc|} \hline 
$q$  & 3 &4&5&7&8&9&11&13 \\ \hline  \hline
Size Max. Coclique & 1 &4 & 4&4 &10 & 8 &12 & 12 \\ \hline
\end{tabular}
\caption{Size of maximum 2-intersecting set in $\psl{q}$ \label{tab:BestPSL}}
\end{table}

Pablo Spiga noted that the set-wise stabilizer of subsets of size two of the projective line in $\psl{q}$ when 
$q \equiv 1 \pmod{4}$ forms a 2-intersecting set of size $q-1$~\cite[Problem 43]{CameronWebsite}. 
Our results indicate that this is indeed the largest such intersecting set.
\begin{conj}
Let $q \equiv 1 \pmod{4}$, then the size of the maximum 2-intersecting set in $\psl{q}$ is $(q-1)$.
\end{conj}

\bibliographystyle{plain}

\end{document}